\documentclass[12pt,a4paper]{amsart}
\usepackage{amsmath}
\usepackage{amsthm}
\usepackage{amssymb}
\usepackage{mathrsfs}
\usepackage[english]{babel}
\usepackage{ot1patch}

\newtheorem{thm}{Theorem}[section]
\newtheorem{lem}[thm]{Lemma}

\newtheorem{prop}[thm]{Proposition}

\theoremstyle{remark}
\newtheorem{rem}[thm]{Remark}
\newtheorem*{rem*}{Remark}

\theoremstyle{definition}

\numberwithin{equation}{section}

\newtheorem*{lem*}{{\rm\bf Lemma}}

\newcommand{\C}{\mathbb{C}}

\newcommand{\ofrac}[2]{\genfrac{}{}{0pt}{}{#1}{#2}}

\begin{document}

\title{The complex gradient inequality with parameter}

\author{Maciej P. Denkowski}\address{Jagiellonian University, Faculty of Mathematics and Computer Science, Institute of Mathematics, ul. \L ojasiewicza 6, 30-348 Krak\'ow, Poland}\email{maciej.denkowski@uj.edu.pl}
\date{March 4th 2014, {\it Revised:} April 4th 2014}

\begin{abstract}
We prove that given a holomorphic family of holomorphic functions with isolated singularities at zero and constant Milnor number, it is possible to obtain the gradient inequality with a uniform exponent.
\end{abstract}

\maketitle


\section{Introduction}

This note was inspired by a recent article of A. P\l oski \cite{P2} concerning the semi-continuity of the \L ojasiewicz exponent in a family of multiplicity-constant deformation of a finite holomorphic germ (actually, our main tool is a method of P\l oski used already in \cite{P}). Of course, it would probably be possible to derive the main result presented herein from this article, but it seems interesting to give a direct, elementary and self-consistent proof, since it is a starting point for tackling a much more general problem. 

We are interested in the \L ojasiewicz Gradient Inequality: given an (complex or real) analytic (or subanalytic and of class $\mathcal{C}^1$) function $f$ of $m$ variables and such that $f(0)=0$ we are able to find an exponent $\theta\in(0,1)$ such that in a neighbourhood of zero
$$
|f(x)|^\theta\leq\mathrm{const.}||\nabla f(x)||
$$
where $\nabla f(x)$ denotes the (complex or real) gradient of $f$. The classical version (for a real analytic function) was obtained by \L ojasiewicz from his famous \L ojasiewicz inequality that he used to solve L. Schwartz's Division Problem. One of the simplest and most important applications of this inequality is the study of the analytic gradient dynamics, see \cite{L2} (the fact that $\theta\in(0,1)$ is of utmost importance; to make the paper self-consistent, we give a simple proof of this in our special case in Section 2). 

Our general aim is to obtain a parameter version of this inequality, i.e. to prove that given a well-parametrized family of functions we can find (locally) a uniform exponent for this inequality to be satisfied by each of these functions. Probably the simplest case is presented in this paper and it concerns a $\mu$-constant holomorphic unfolding. It may be particularly interesting e.g. in view of the results in \cite{GS}.

A general subanalytic parameter version of the gradient inequality will be given in a forthcoming paper.

\section{On the gradient exponent}

For convenience sake, we will give --- using only complex analytic geometry tools in the spirit of P\l oski --- yet another and elementary proof of the gradient inequality for a holomorphic germ $f\colon (\mathbb{C}^m, 0)\to (\mathbb{C}, 0)$ satisfying the condition $\dim_0\nabla f^{-1}(0)=0$, where $$\nabla f=\left(\frac{\partial f}{\partial z_1},\dots, \frac{\partial f}{\partial z_m}\right)\colon \mathbb{C}^m_z\to\mathbb{C}^m_w$$ is the complex gradient. What will be of greatest interest for our purposes is the bound on the exponent we obtain. 

We may assume $m>1$, since the case $m=1$ is a simple exercise involving power series. Hereafter we will insist on details that for simplicity we omit in the next section.

Our approach is based on \cite{P}. Take a connected neighbourhood $U$ of $0\in \mathbb{C}^m$ for which $\nabla f^{-1}(0)\cap U=\nabla f^{-1}(0)\cap\overline{U}=\{0\}$. Hence $\nabla f$ is a proper map from $U$ onto a connected neighbourhood $V$ of zero. This means that the projection $\pi(z,w)=w$ is proper on $\Gamma_{\nabla f}\cap U\times V$ and sends this part of the graph onto $V$. It is thence a branched covering (see \cite{Ch}) with covering number $\mu$. Its critical locus\footnote{I.e. the set of points $w\in V$ for which $\#\nabla f^{-1}(w)<\mu$ -- in this case it coincides with the critical values in the usual sense.} $\sigma\subsetneq V$ is analytic and $\mu$ is in fact the \textit{Milnor number} of $f$ at zero. For $w\in V\setminus \sigma$ we have $\nabla f^{-1}(w)=\{z^{(1)},\dots, z^{(\mu)}\}$ consisting of exactly $\mu$ points and so we can define for $t\in\mathbb{C}$
\begin{align*}
P(w,t):&=\prod_{z\in\nabla f^{-1}(w)}(t-f(z))=\\
&=t^\mu+a_1(w)t^{\mu-1}+\ldots+a_\mu(w),\\
&\textrm{where}\quad a_j(w)=(-1)^j\sum_{1\leq \iota_1<\ldots<\iota_j\leq \mu} f(z^{(\iota_1)})\cdot\ldots\cdot f(z^{(\iota_j)})
\end{align*}
are holomorphic functions on $V\setminus\sigma$ admitting continuous extensions onto $V$. Therefore, by the Riemann Extension Theorem, $P\in\mathcal{O}(V)[t]$. This polynomial is called \textit{the characteristic polynomial of} $f$ with respect to $\nabla f$. Clearly, $P(\nabla f,f)\equiv 0$ and $a_j(0)=0$. 

\begin{rem}
Note that in the same way we can define the characteristic polynomial of $f$ with respect to any holomorphic map germ $g\colon ({\C}^m,0)\to ({\C}^m,0)$ satisfying $g^{-1}(0)=\{0\}$.
\end{rem}

The Vi\`ete formul{\ae} yield $|f(z)|\leq 2\max_j|a_j(\nabla f(z))|^{1/j}$. On the other hand, since in a neighbourhood of zer there is $|a_j(w)|\leq c_j||w||^{\mathrm{ord}_0 a_j}$ for some $c_j>0$ ($\mathrm{ord}_0 h$ denotes the order of vanishing at zero\footnote{$\mathrm{ord}_0 h$ is the degree of the initial form in the expansion of $h\not\equiv 0$ into homogeneous forms near zero. There is $\mathrm{ord}_0 h=\max\{\eta>0\mid |h(z)|\leq \mathrm{const.}||z||^\eta$ in a neighbourhood of $0\}$, see \cite{Ch}.} of the holomorphic function $h$), this leads to
$$
|f(z)|\leq \mathrm{const.}||\nabla f(z)||^{\min_{j=1}^\mu\left(\frac{\mathrm{ord}_0 a_j}{j}\right)}
$$
in a neighbourhood of zero. We have obtained the gradient inequality with the exponent
$$
\theta=\max_{j=1}^\mu \frac{j}{\mathrm{ord}_0 a_j}
$$
and we shall prove in addition that $\theta<1$ (i.e.  $\mathrm{ord}_0 a_j> j$ for all $j$). This is still rather delicate (see the crucial article \cite{T} of B. Teissier, Corollaire 2 p. 270 and compare with \cite{PW} Theorem 2). 

Geometrically, the weaker condition:
$$\mathrm{ord}_0 a_j\geq j\ \textrm{for all}\ j,$$ is equivalent to saying that the tangent cone of $P^{-1}(0)$ at zero, described by the zeroes of the initial form $\mathrm{in} P$ in the expansion of $P$ into homogeneous forms at zero meets $\{0\}^m\times\mathbb{C}$ only at zero (since this means precisely that $\mathrm{in} P(0,t)=t^\mu$; actually, we will show that $\mathrm{in} P(w,t)=t^\mu$). Recall the Peano tangent cone:
$$
C_0(P^{-1}(0))=\{v\in\mathbb{C}^{m+1}\mid \exists P^{-1}(0)\ni x_\nu\to 0, \exists \lambda_\nu>0\colon \lambda_\nu x_\nu\to v\}.
$$

The weak version follows from \cite{P}, as $P(\nabla f, f)=0$ is an equation of integral dependence of $f$ over the ideal $\left\langle \frac{\partial f}{\partial z_1},\dots, \frac{\partial f}{\partial z_m}\right\rangle$ which implies (actually, is equivalent to) the inequality $$|f(z)|\leq\mathrm{const.} ||\nabla f(z)||\leqno{(\star)}$$ in a neighbourhood of zero\footnote{From the expressions for $a_j$ and the Vi\`ete formul{\ae} it is easy to see that $P(\nabla f,f)\equiv 0$ implies $|f(z)|\leq\mathrm{const.}||\nabla f(z)||$. For the other way round it suffices to observe that $a_j$ belongs to $\mathfrak{m}^j=\langle z^\alpha\colon |\alpha|=j\rangle$, where $\mathfrak{m}$ is the maximal ideal of the ring $\mathcal{O}_m$.}. Now, since $P^{-1}(0)$ and $(\nabla f,f)(U)$ define the same set germ at zero, we see that for any $v=(0,t)\in C_0(P^{-1}(0))$ we can find sequences $U\ni z_\nu\to 0$ and $\lambda_\nu\to 0^+$ such that $\lambda_\nu \nabla f(z_\nu)\to 0$ and $\lambda_\nu f(z_\nu)\to t$. Therefore, $(\star)$ implies $t=0$. 
We need only to prove that $\mathbb{C}^m\times\{0\}\supset C_0(P^{-1}(0))$ (for we know that $C_0(P^{-1}(0))$ is an algebraic cone of dimension $\dim_0 P^{-1}(0)=m$). 
We have just seen that $C_0(P^{-1}(0))\cap \{0\}^m\times\mathbb{C}=\{0\}^{m+1}$. 
Shrinking the neighbourhoods we may assume that $P^{-1}(0)=(\nabla f,f)(U)$ and this has pure dimension $m$. Then we can use the elementary Proposition from \cite{Ch} p. 86: for any $(w,t)\in C_0(P^{-1}(0))\setminus \{(0,0)\}$ (in particular for $w\neq 0$) we have 
$$
(w,t)\in C_0(P^{-1}(0)\cap (\mathbb{C}w\times\mathbb{C}))=:C.
$$
But, $P^{-1}(0)\cap (\mathbb{C}w\times\mathbb{C})=\{(\nabla f(z), f(z))\mid z\in \nabla f^{-1}(\mathbb{C}w)\}$ is a curve as the image by $(\nabla f, f)$ of the polar curve 
$X:=\nabla f^{-1}(\mathbb{C}w)$ ($\nabla f$ is surjective and we are in a neighbourhood of the origin). Fix $(w,t)$ and denote by $\Gamma$ an irreducible component of $(\nabla f, f)(X)$ for which $(w,t)\in C_0(\Gamma)$. Then there exists  $v\in\mathbb{C}^{m+1}\setminus\{0\}$ such that $C_0(\Gamma)=\mathbb{C}v$. Now we need the following lemma.
\begin{lem}
Let $h\colon W\to \mathbb{C}^n$ be a holomorphic function defined in a neighbourhood $W\subset\mathbb{C}$ of zero, satisfying $h(0)=0$ and such that $h(W)$ defines an irreducible curve germ at zero. Then there is a sequence $W\ni t_\nu\to 0$ such that $h(t_\nu)/||h(t_\nu)||$ converge to a vector $v$ for which $C_0(h(W))=\mathbb{C}v$.
\end{lem}
\begin{proof}[Proof of the Lemma]
Take a sequence $W\setminus h^{-1}(0)\ni t_\nu\to 0$. Since $\{h(t_\nu)/||h(t_\nu)||\}\subset\mathbb{S}^{2n-1}$, extcracting, if necessary, a subsequence, we may assume that these terms converge to some $v\in\mathbb{S}^{2n-1}$. But $h(t_\nu)\to 0$ and $\lambda_\nu=1/||h(t_\nu)||>0$. Hence $v\in C_0(h(W))$, whence the cone contains the line $\mathbb{C}v$. The germ at zero of $h(W)$ being irreducible, the tangent cone coincides with the line found.
\end{proof}

Let $\Gamma'\subset X$ be an irreducible component such that $(\nabla f, f)(\Gamma')=\Gamma$. Consider a Puiseux parametrization $\gamma\colon W\to \Gamma'$. Then $h:=(\nabla f, f)\circ \gamma$ satisfies the assumptions of the preceding lemma and so we have a sequence $t_\nu\to 0$ such that $v=\lim h(t_\nu)/||h(t_\nu)||$. Write $v=(v', v'')\in\mathbb{C}^m\times\mathbb{C}$; we will show that $v''=0$. We look at $f(\gamma(t_\nu))$; on the one hand,
$$
|f(\gamma(t_\nu))|\leq\mathrm{const.}|t_\nu|^{\mathrm{ord}_0(f\circ\gamma)}, \> \nu\gg 1,
$$
whereas on the other,
$$
||h(t_\nu)||\geq ||\nabla f(\gamma(t_\nu))||\geq\mathrm{const.}|t_\nu|^{\mathrm{ord}_0(\nabla f\circ\gamma)},\> \nu\gg 1,
$$
where obviously\footnote{From the expansion into a power series it is easy to see that for a non-constant analytic germ $g\colon (\mathbb{C}, 0)\to(\mathbb{C}, 0)$ there is $c_1|t|^q\leq|g(t)|\leq c_2|t|^q$ in a neighbourhood of zero with $q=\mathrm{ord}_0 g$, $c_1, c_2>0$. Hence, if $g=(g_1,\dots, g_n)$ on $(\mathbb{C}, 0)$ the same kind of inequality holds true with $q=\min_j\mathrm{ord}_0 g_j$.} $\mathrm{ord}_0(\nabla f\circ\gamma)=\min_j\mathrm{ord}_0(\frac{\partial f_j}{\partial z}\circ \gamma)$. For $\nu$ large enough,
$$
\frac{|f(\gamma(t_\nu))|}{||h(t_\nu)||}\leq\mathrm{const.} |t_\nu|^{\mathrm{ord}_0(f\circ\gamma)-\mathrm{ord}_0(\nabla f\circ\gamma)}
$$
and it remains to observe that the exponent is positive. 

To see this we argue classically as follows: if $g\colon (\mathbb{C}, 0)\to(\mathbb{C}, 0)$ is an analytic germ, then $\mathrm{ord}_0 g=\mathrm{ord}_0 g'+1$, whence $\mathrm{ord}_0(f\circ\gamma)=\mathrm{ord}_0(f\circ\gamma)'+1$. But (since $\mathrm{ord}_0\sum_1^k g_j\geq\min_j\mathrm{ord}_0 g_j$)
\begin{align*}
\mathrm{ord}_0(f\circ\gamma)'&=\mathrm{ord}_0\langle \nabla f\circ\gamma, \gamma'\rangle\geq \\
&\geq \min_j\mathrm{ord}_0 \left(\frac{\partial f}{\partial z_j}\circ\gamma\right)\gamma_j'=\\
&=\min_j\left\{\mathrm{ord}_0\left(\frac{\partial f}{\partial z_j}\circ\gamma\right)+\mathrm{ord}_0\gamma_j'\right\}.
\end{align*}
Since $\mathrm{ord}_0\gamma'_j\geq 0$ for all $j$, the last expression is greater or equal to $\mathrm{ord}_0(\nabla f\circ \gamma)$.

Summing up, we have proved the following Proposition (cf. \cite{T} Corollaire 2 p. 270 and \cite{P}, \cite{PW}):
\begin{prop}\label{wykladnik}
If $f\in\mathcal{O}_m$ is a holomorphic germ such that $\nabla f$ has an isolated zero at the origin, then the coefficients of the characteristic polynomial $P(w,t)=t^\mu+a_1(w)t^{\mu-1}+\ldots+a_\mu(w)\in \mathcal{O}_m[t]$, of $f$ with respect to $\nabla f$, where $\mu$ is the Milnor number of $f$, all satisfy $\mathrm{ord}_0 a_j\geq j+1$ and we have in a neighbourhood of zero the gradient inequality
$$
|f(z)|^{\max_{j=1}^\mu \frac{j}{\mathrm{ord}_0a_j}}\leq\mathrm{const.} ||\nabla f(z)||.
$$
\end{prop}

\section{Gradient inequality with parameter}

Let $f(x,t)\in\mathcal{O}_{m+k}$ be a holomorphic germ. Write $f_t(x):=f(x,t)$ and put $g_t(x):=\nabla f_t(x)$ for the complex gradient. The resulting germ $g(x,t):=g_t(x)$ is, of course, holomorphic.  Throughout this section we assume that $f(0,t)\equiv 0$ (which is not really restrictive).

We will assume also that for all $t$ small enough, $g_t^{-1}(0)=\{0\}$ is isolated at zero. Furthermore, we will assume that the Milnor numbers $m_0(g_t)\equiv \mu$ are \textit{constant}.

However, before we do that, let us have a closer look at the general situation. If we just assume that $g_0^{-1}(0)=\{0\}$ is isolated, then the holomorphic germ $G(x,t):=(g(x,t),t)$ has an isolated zero at the origin (since $G^{-1}(0)=g_0^{-1}(0)\times\{0\}$). Therefore, we may consider it as a branched covering between domains\footnote{They can be chosen to be polydiscs.} $U\times V, W\subset{\C}^{m+k}$ around zero, of geometric multiplicity\footnote{To state it clearly: the geometric multiplicity of a branched covering $h$, denoted by $m(h)$, is its sheet number, whereas the geometric multiplicity of it at a point $x_0$, denoted $m_{x_0}(h)$, is the generic number of points converging to $x_0$ from nearby fibres.} $m(G)$ coinciding with the geometric multiplicity $m_0(G)$ at the origin. Put $\mu:=m_0(G)$.

Observe that $G^{-1}(y,t)=g_t^{-1}(y)\times\{y\}$ and in particular all the $g_t$, $t\in V$, are branched coverings $U\to W_t$, where $W_t=\{x\in{\C}^m\mid (x,t)\in W\}$. However, $m(g_t)\geq m_0(g_t)$ where the latter is precisely the Milnor number of $f_t$ in case $g_t(0)=0$. (Of course, our assumptions imply that at least $m_0(g_0)=m(g_0)$.) To be more precise, it is a classical result due presumably to W. Stoll that for any branched covering $h$ between domains in ${\C}^N$, $$m(h)=\sum_{x\in h^{-1}(h(x_0))} m_x(h)\leqno{(\#)}$$ 

Clearly, $m(g_t)\leq\mu$. Actually, it turns out that $m(g_t)=\mu$, for all $t\in V$. In order to see this, we will look first at the critical locus $\sigma$ of $G$, which is precisely $G(C_G)$ where $C_G=\{(x,t)\in U\times V\mid \det d_{(x,t)} G=0\}$ is the set of critical points. Observe that 
$$
\det d_{(x,t)} G=\det \frac{\partial g}{\partial x}(x,t)=\det\frac{\partial g_t}{\partial x}(x).
$$
But we know that the critical locus $\sigma_t$ of $g_t$ is precisely $g_t(C_{g_t})$ and so we conclude that 
$$
x\in\sigma_t\ \Leftrightarrow\ (x,t)\in\sigma.
$$
In particular, this implies that the sections $\sigma\cap ({\C}^m\times\{t\})$ are nowheredense in $W\cap  ({\C}^m\times\{t\})$ whence we conclude that $m(g_t)=\mu$ for all $t\in V$. Indeed, for each $t$, there is always a point $y$ near zero that is critical neither for $G$ (i.e. $(y,t)\notin W\setminus \sigma$), nor for $g_t$ (i.e. $y\notin \sigma_t$) and so
$$
\mu=\#G^{-1}(y,t)=\#(g_t^{-1}(y)\times\{t\})=\#g_t^{-1}(y)=m(g_t).
$$

All this leads to the following simple observation that should be compared with condition (6) from the L\^e-Saito-Teissier criterion of $\mu$-constancy for one-parameter unfoldings as exposed in \cite{G} p. 161:
\begin{prop}\label{GreuelGen}
Keeping the notations introduced so far, assume that $f\in\mathcal{O}_{m+k}$ is such that $g_t(0)=0$ for $|t|\ll 1$ and that $g_0^{-1}(0)=\{0\}$ is isolated. Then \begin{enumerate} 
\item For all $t$ sufficiently close to zero $g_t^{-1}(0)$ is isolated too and the resulting branched coverings are all $\mu$-sheeted, where $\mu=m_0(g_0)$ is the Milnor number of $g_0$; in particular the Milnor numbers of $g_t$ are not greater than $\mu$;
\item The unfolding $f$ of $f_0$ is $\mu$-constant iff in a neighbourhood of zero we have
$$
\left\{(x,t)\in{\C}^m\times{\C}^k\colon\frac{\partial f}{\partial x_j}(x,t)=0, j=1,\dots, m\right\}=\{0\}^m\times{\C}^k.
$$
\end{enumerate}
\end{prop}
\begin{proof}
The first part of the assertion follows from the preceding discussion. The second part is a consequence of $(\#)$. Indeed, we already know that $m(g_t)=\mu$ and we assumed that $g_t(0)=0$. Thus,  the formula in question implies that $m_0(g_t)=\mu$ iff $\#g_t^{-1}(0)=1$. This gives (2).
\end{proof}

After this introduction, the main theorem we are aiming at is the following:
\begin{thm}
If $f(x,t)\in\mathcal{O}_{m+k}$ is such that for all $t$ small enough, $f_t(x)$ has an isolated singularity at zero with constant Milnor number $\mu$, then for all $(x,t)$ in a neighbourhood of zero
 $$
|f_t(x)|^{\frac{\mu}{\mu+1}}\leq  C ||\nabla f_t(x)||
$$
where $C>0$. Note that the exponent belongs to $[1/2,1)$ and $f$ itself satisfies the gradient inequality (with respect to $(x,t)$) with this exponent.
\end{thm}
\begin{proof}
The proof will be given in several steps.

\textbf{Step 1}

Consider as earlier the proper mapping germ $G(x,t):=(g(x,t),t)$, 
$G\colon ({\C}^{m+k},0)\to({\C}^{m+k},0)$.

By Proposition \ref{GreuelGen} we know that the multiplicity at zero of the branched covering $G$ is $\mu$.


\textbf{Step 2}

Let now $P(y,t,s)\in\mathcal{O}_{m+k}[s]$ be the characteristic polynomial of $f$ with respect to $G$. Therefore, we have $P(t,g(x,t),f(x,t))\equiv 0$ and so by the Vi\`ete formul{\ae}:
$$
|f(x,t)|\leq 2\max_{j=1}^\mu |a_j(g(x,t),t)|^{1/j}
$$
where we write $P(y,t,s)=s^\mu+a_1(y,t)s^{\mu-1}+\ldots+a_\mu(y,t)$. 
\begin{lem}
$P(\cdot,t,\cdot)$ is 
the characteristic polynomial of $f_t$ with respect to $g_t$.
\end{lem}
\begin{proof}
For a generic point $y$ near zero there is $\#g_t^{-1}(y)=\mu$. Let $x^{(1)},\dots, x^{(\mu)}$ be the points in this fibre. Then since the fibre $G^{-1}(y,t)$ consists precisely of the points $(x^{(j)},t)$, it is maximal and we get 
$$
P(y,t,s)=\prod_{j=1}^\mu (s-f(x^{(j)},t))=\prod_{j=1}^\mu (s-f_t(x^{(j)}))
$$
which gives the assertion.
\end{proof}

In particular we have $a_j(0,t)=0$.

\textbf{Step 3}

Fix $j$ and consider the expansion of $a_j$ into a Hartogs series in a polydisc $P(r)$ centred at zero (the same for all $j=1,\dots,\mu$)
$$
a_j(y,t)=\sum_{\nu=\nu_j(t)}^{+\infty} \sum_{\ofrac{\alpha\in\mathbb{Z}_+^m}{|\alpha|=\nu}} \left(\frac{1}{\alpha!}D^{(\alpha,0)}a_j(0,t)\right)y^\alpha
$$
where $\nu_j(t):=\mathrm{ord}_0 a_j(\cdot,t)\geq j+1$ by Proposition \ref{wykladnik}. Write 
$$
b_{j,\nu}^t(y)=\sum_{\ofrac{\alpha\in\mathbb{Z}_+^m}{|\alpha|=\nu}} \left(\frac{1}{\alpha!}D^{(\alpha,0)}a_j(0,t)\right)y^\alpha
$$
for the $\nu$-homogeneoous term. For $(y,t)\in P(r)$ fixed and $\lambda\in{\C}$ we have
$$
a_j\left(\frac{y}{||y||}\lambda,t\right)=\sum_{\nu=\nu_j(t)}^{+\infty} \left(\frac{b_{j\nu}^t(y)}{||y||^\nu}\right)\lambda^\nu,\> \textrm{provided}\> |\lambda|<r.
$$

The Cauchy inequalities yield
$$
\left\vert \frac{b_{j\nu}^t(y)}{||y||^\nu}\right\vert\leq\sup_{|\lambda|<r}\frac{\left\vert a_j\left(\frac{y}{||y||}\lambda,t\right)\right\vert}{r^\nu}\leq \frac{M_j}{r^\nu}.
$$
with $M_j:=\sup\{|a_j(y,t)|\colon (y,t)\in P(r)\}>0$. Let $M=\max_{j=1}^\mu M_j$. Writing $j+1$ instead of $\nu_j(t)$ we have in $P(r)$,
\begin{align*}
|a_j(y,t)|&\leq \sum_{\nu=j+1}^{+\infty} \frac{M}{r^\nu} ||y||^\nu\leq\\
&\leq ||y||^{j+1} \frac{M}{r^{j+1}}\sum_{\nu=0}^{+\infty} \left(\frac{||y||}{r}\right)^{\nu}.
\end{align*}
This leads to the conclusion that for some constant $C>0$ we have for all $(y,t)$ in a neighbourhood $V$ of zero and all indices $j$,
$$
|a_j(y,t)|\leq C||y||^{j+1}.
$$

\textbf{Step 4}

Note that $g(0,t)=0$ and we obtain the inequalities
$$
|a_j(g(x,t),t)|\leq C||g(x,t)||^{j+1}.
$$
for all $j$ and $(x,t)$ in a neighbourhood $U$ of zero such that $G(U)\subset V$. 

This gives the gradient inequality with parameter and
$$
\max_{j=1}^\mu\frac{j}{j+1}=\frac{\mu}{\mu+1}
$$
as exponent. The proof is complete.
\end{proof}

\begin{rem}
In view of the preceding discussion, if the unfolding is not $\mu$-constant, we have at least the inequality $\mu_t\leq\mu_0$ for $t$ close to zero, where $\mu_t=m_0(g_t)$ is the Milnor number of $f_t$. Since the function $s\mapsto \frac{s}{s+1}$ is increasing and we may always assume that we are in a neighborhood in which $|f(t,x)|<1$, we easily obtain that 
$$
|f_t(x)|^{\frac{\mu_0}{\mu_0+1}}\leq C_t||\nabla f_t(x)||
$$
in a neighbourhood $V_t$ of zero that will depend on $t$ (cf. Proposition \ref{GreuelGen}) just as the constant $C_t$. 
\end{rem}

\section{Acknowledgements}

This research was partially supported by Polish Ministry of Science and Higher Education grant 1095/MOB/2013/0.

The author would like to thank A. P\l oski for the reference \cite{T} and M. Tib\u ar for attracting his attention to the paper \cite{G}. Many thanks are due to the University of Lille 1 for excellent working conditions.

\end{document}